\documentclass[12pt,reqno]{amsart}

\newcommand\version{November 23, 2023}

\usepackage{amsmath, amsfonts, amsthm, amssymb, amsxtra}
\usepackage{bbm} % to get \1
\usepackage{hyperref}
\usepackage{xcolor}

\newtheorem{theorem}{Theorem}%[section]
\newtheorem{proposition}[theorem]{Proposition}
\newtheorem{lemma}[theorem]{Lemma}
\newtheorem{corollary}[theorem]{Corollary}

\theoremstyle{definition}

\theoremstyle{remark}

\newtheorem{remark}[theorem]{Remark}
\newtheorem{remarks}[theorem]{Remarks}

\newcommand{\1}{\mathbbm{1}}

\newcommand{\C}{\mathbb{C}}

\renewcommand{\epsilon}{\varepsilon}

\newcommand{\N}{\mathbb{N}}

\renewcommand{\phi}{\varphi}
\newcommand{\R}{\mathbb{R}}

\newcommand{\Sph}{\mathbb{S}}

\DeclareMathOperator{\Div}{div}

\DeclareMathOperator{\re}{Re}

\DeclareMathOperator{\supp}{supp}

\DeclareMathOperator{\HyperF}{F}

\newcommand{\Hypergeom}[5]{{\sideset{_#1}{_#2}\HyperF\!\left(
		\genfrac{}{}{0pt}{0}{#3}{#4}
		\middle|\,#5\right)}}

\newcommand{\rF}{\mathrm{F}}
\newcommand{\HG}[2]{{}_{#1}\rF_{#2}} % hypergeom

\begin{document}

\title[Minimizers for an aggregation model --- \version]{Minimizers for an aggregation model\\ with attractive--repulsive interaction}

\author{Rupert L. Frank}
\address[Rupert L. Frank]{Mathe\-matisches Institut, Ludwig-Maximilans Universit\"at M\"unchen, The\-resienstr.~39, 80333 M\"unchen, Germany, and Munich Center for Quantum Science and Technology, Schel\-ling\-str.~4, 80799 M\"unchen, Germany, and Mathematics 253-37, Caltech, Pasa\-de\-na, CA 91125, USA}
\email{r.frank@lmu.de}

\author{Ryan W. Matzke}
\address[Ryan W. Matzke]{Department of Mathematics, Vanderbilt University, 1326 Stevenson Center, Station B 407807, Nashville, TN 37240, USA}
\email{ryan.w.matzke@vanderbilt.edu}

\thanks{\copyright\, 2023 by the authors. This paper may be reproduced, in its entirety, for non-commercial purposes.\\
	R.L.F.\ is partially supported through US National Science Foundation grant DMS-1954995, as well as through the Deutsche Forschungsgemeinschaft Excellence Strategy EXC-2111-390814868 and TRR 352 (Project-ID 470903074). R.W.M.\ is supported by US National Science Foundation Postdoctoral Research Fellowship Grant 2202877.}

\begin{abstract}
	We solve explicitly a certain minimization problem for probability measures involving an interaction energy that is repulsive at short distances and attractive at large distances. We complement earlier works by showing that in an optimal part of the remaining parameter regime all minimizers are uniform distributions on a surface of a sphere, thus showing concentration on a lower dimensional set. Our method of proof uses convexity estimates on hypergeometric functions.
\end{abstract}

\maketitle

\section{Introduction and main result}

We are interested in a minimization problem that arises in mathematical biology, mathematical physics, and economics as a toy model for aggregation. It involves a mean-field model of particles, or individuals, interacting through a pair potential, whose resulting force is repulsive for short distances and attractive for long distances. The short range repulsion restricts the collision of particles and the long range attraction prevents particles from spreading out. One then expects the interaction energy to achieve its minimal value at certain steady/stable states of the system, and one interprets this as an example of self-organization and pattern formation. These phenomena and related theory have been the subject of several recent works, many of which we mention below. We would like to determine the explicit equilibrium states of the interaction energy for a particular model, which depends on two parameters related to the nature of the repulsive/attractive forces. The characterization of minimal energy states has previously been achieved for a certain ranges of these parameters. Our contribution in this paper is to enlarge this regime where the states of minimal energy are known. In particular, we will be able to identify critical parameter values, where a phase transition occurs.

Let us now introduce the model and present our results.  We work in spatial dimension $d\geq 1$ and denote by $P(\R^d)$ the set of Borel probability measures on $\R^d$. For parameters $-d<\beta<\alpha<\infty$, we consider the following functional, defined for $\mu\in P(\R^d)$,
$$
\mathcal E_{\alpha,\beta}[\mu] :=
	\frac12 \iint_{\R^d\times\R^d} \left( \alpha^{-1}|x-y|^\alpha - \beta^{-1}|x-y|^\beta \right) d\mu(x)\,d\mu(y) \,.
$$
Here we use the convention to interpret $\gamma^{-1}|x-y|^\gamma$ for $\gamma=0$ as $\ln |x-y|$. The minimal energy is
$$
E_{\alpha,\beta} := \inf\left\{ \mathcal E_{\alpha,\beta}[\mu] :\ \mu\in P(\R^d) \right\}
$$
and we are interested in determining $\mu \in P(\R^d)$ such that
$$
\mathcal{E}_{\alpha, \beta}(\mu) =  E_{\alpha,\beta} \,.
$$

A special role will be played by the parameter value
$$
\beta_*(\alpha) := \frac{-10+3\alpha+7d-\alpha d- d^2}{d+\alpha-3} \,,
$$
defined for $\alpha\geq 2$ and $d\geq 2$. This function is decreasing in $\alpha$ and satisfies for $\alpha \geq 2$, $-d+3 <\beta_*(\alpha)\leq \beta_*(2) = -d+4$. It will be convenient to introduce a number $R_{\alpha,\beta}$ defined, for $\beta_*(\alpha)\leq\beta<\alpha$, by
\begin{equation}
	\label{eq:defr1}
	R_{\alpha,\beta} :=
	\frac12 \left( \frac{\Gamma(\frac{d+\beta-1}{2})\, \Gamma(\frac{2d+\alpha-2}{2})}{\Gamma(\frac{d+\alpha-1}{2})\, \Gamma(\frac{2d+\beta-2}{2})} \right)^\frac{1}{\alpha-\beta}
\end{equation}
and, for $\alpha=2$ and $-d<\beta<-d+4$, by
\begin{equation}
	\label{eq:defr2}
	R_{2,\beta} :=
	\left( \frac{\Gamma(\frac{4-\beta}2)\,\Gamma(\frac{\beta+d}{2})}{\Gamma(1+\frac d2)} \right)^\frac{1}{2-\beta}.
\end{equation}
Also, we will denote by $B_R(a)$ the open ball of radius $R$ centered at $a$.

We now state our main result.

\begin{theorem}\label{main1}
	Let $d\geq2$, $2\leq\alpha\leq 4$ and $\beta_*(\alpha) \leq\beta\leq 2$ with $\beta<\alpha$. Then
	$$
	E_{\alpha,\beta} = 
	\begin{cases}
		- \pi^{-\frac12} \, 2^{d+\alpha-3}\, \frac{\Gamma(\frac d2)\,\Gamma(\frac{d+\alpha-1}{2})}{\Gamma(\frac{2d+\alpha-2}{2})} \left( \frac1\beta - \frac1\alpha \right) R_{\alpha,\beta}^\alpha 
		& \text{if}\ \beta\neq 0 \,,\\
		\frac1{2\alpha}\left( 1 - \ln \frac{\Gamma(\frac{d-1}2)\,\Gamma(\frac{2d+\alpha-2}{2})}{\Gamma(d-1)\,\Gamma(\frac{d+\alpha-1}{2})} \right) + \frac14 \left( \frac{\Gamma'(d-1)}{\Gamma(d-1)} - \frac{\Gamma'(\frac{d-1}{2})}{\Gamma(\frac{d-1}{2})} \right) & \text{if}\ \beta =0 \,.
	\end{cases}
	$$
	The infimum is attained if for some $a\in\R^d$,
	$$
	\mu = (|\Sph^{d-1}|R_{\alpha,\beta}^{d-1})^{-1}\ \delta_{\partial B_{R_{\alpha,\beta}}(a)} \, ,
	$$
	Provided $(\alpha,\beta)\neq (4,2)$, these are the only minimizers of $\mathcal{E}_{\alpha, \beta}$.
\end{theorem}

This result is sharp in the sense that, for $2\leq\alpha\leq 4$, if $\beta<\alpha$ does \emph{not} belong to the interval $[\beta_*(\alpha),2]$, then the uniform distribution on the surface of a sphere is \emph{not} a minimizer; see item (d) in the remarks after the next theorem. Meanwhile, it is conceivable that a similar result also holds outside the range $2\leq\alpha\leq 4$ for suitable values of $\beta$. 

We believe that the result of Theorem \ref{main1} is new, except in the particular cases $\beta=2$ and $(\alpha,\beta)=(2,-d+4)$, for which we provide references after the next theorem. We refer to \cite{DaLiMC2} for a classification of minimizers in the case $(\alpha,\beta)=(4,2)$ that we have excluded.

We use the method of proof of Theorem \ref{main} to give an alternative proof of the following known result concerning $\alpha=2$ and $\beta<\beta_*(2)=-d+4$. (The latter equality can be viewed as the definition of $\beta_*(2)$ for $d=1$.)

\begin{theorem}\label{main}
	Let $d\geq 1$, $\alpha =2$, and  $-d<\beta < \min\{2, -d+4\}$.
	Then
	$$
	E_{2,\beta} = 
	\begin{cases}
		- \frac{d (2-\beta)}{2\beta(4-\beta)} \, R_{2,\beta}^2 & \text{if}\ \beta\neq 0 \,,\\
		\frac14\left( \frac12 + \ln\frac d2 + \frac{\Gamma'(2)}{\Gamma(2)} - \frac{\Gamma'(\frac d2)}{\Gamma(\frac d2)} \right) & \text{if} \ \beta = 0 \,.
	\end{cases}
	$$
	The infimum is attained if and only if, for some $a\in\R^d$,
	$$
	d\mu(x) = C_\beta^{-1} \ R_{2,\beta}^{\beta-2} \ (R_{2,\beta}^2 - |x-a|^2)^{\frac{2-\beta-d}2} \ \1_{B_{R_{2,\beta}}(a)}(x)\,dx
	$$
	with
	\begin{equation}
		\label{eq:defca}
		C_\beta := \pi^\frac d2\, \frac{\Gamma(\frac{4-\beta-d}{2})}{\Gamma(\frac{4-\beta}{2})} \,.
	\end{equation}
\end{theorem}

\begin{remarks}\label{mainrem} 
	Here we collect some notes as well as results from previous works.\\
	(a) As mentioned earlier, Theorem \ref{main} (where $\alpha=2$) is a known result. The result in the case $\beta=-d+2$ of Coulomb repulsion is folklore. The result for $-d<\beta<-d+2$ can be extracted from the analysis of the porous medium equation with fractional diffusion by Caffarelli and V\`azquez \cite{CaVa}; for the case $d=1$, see also \cite{Fr1}. The result for $-d+2<\beta<\min\{-d+4,2\}$ is due to Carrillo and Shu \cite[Theorem 5.1]{CaSh1}.\\
	(b) The particular case $(\alpha,\beta)=(2,-d+4)$ of Theorem \ref{main1} has been addressed before. Specifically, in \cite[Remark 5.8]{CaSh1} Carrillo and Shu sketch how their method of proof dealing with the regime $-d+2<\beta<\min\{-d+4,2\}$, $\alpha = 2$, allows them to handle the limiting case $\beta=-d+4$ for $d>2$. As far as we understand, this method does not extend to $\beta>-d+4$. In particular, we point out that the statement below \cite[(2.4)]{DaLiMC1} to the effect that the case $-d+4<\beta<2$ is solved in \cite{CaSh1} is incorrect, as confirmed by the authors of \cite{DaLiMC1} in personal communication.\\
	(c) The limiting case $\beta=2$ of Theorem \ref{main1} has been previously addressed. Indeed, the minimization problem $E_{\alpha,2}$ with $\beta=2$ was completely solved in two papers \cite{DaLiMC1,DaLiMC2} by Davies, Lim, and McCann, except for the case $2<\alpha<3$ in $d=1$ treated in \cite{Fr1}. Besides the result stated in Theorem \ref{main1} in this case, it is shown in these papers that for $\alpha>4$ in $d\geq 2$, minimizers are uniform distributions on the vertices of a regular simplex, while for $\alpha=4$ in $d\geq2$ there is a large family of minimizers. The papers \cite{DaLiMC1,DaLiMC2} also contain partial results for $\beta\neq 2$.\\
	(d) For any $2\leq\alpha\leq 4$ the assumption $\beta_*(\alpha)\leq\beta\leq 2$ on the parameter $\beta$ satisfying $\beta<\alpha$ is optimal for the minimizer being the uniform distribution on the surface of a sphere. Indeed, if $\beta>2$, then minimizing measures are supported on finitely many points \cite{CaFiPa} and, if $\beta<\beta_*(\alpha)$, then the uniform distribution on the surface of a sphere is not even a local minimizer \cite{BaCaLaRa}. For the latter point we also refer to Step 1 in the proof of Lemma \ref{convexsmall} below.\\
	(e) For $\alpha=2$ (where $\beta_*(\alpha)=-d+4$), as the parameter $\beta$ increases, one sees a transition from probability measures that are absolutely continuous with respect to Lebesgue measure with bounded density (for $-d < \beta\leq-d+2$) to such with unbounded density (for $-d+2<\beta<-d+4$) and further, if $d\geq 3$, to a singular measure, namely uniform measure on the surface of a sphere (for $-d+4\leq\beta<2$). Similar transitions are expected for any fixed $2\leq\alpha\leq 4$. One may wonder whether the transition points are always $-d+2$ and $\beta_*(\alpha)$. For partial results for $\alpha=4$ see \cite[Theorem 5.1]{CaSh1}, and for numerical results for $\alpha=4$ in dimension $d=1$ see \cite{GuCaOl}. One can observe a similar transition to more and more singular measures if one fixes $\beta=2$ and lets $\alpha$ increase; see item (c).\\
	(f) A very brief comment on the proof: using the positive definiteness of the kernel $|x-y|^{-\gamma}$ for $\gamma\leq 0$ (or conditional negative definiteness for $\gamma\leq 2$), as well as a certain conditional positive definiteness for $2\leq\gamma\leq 4$, we see that it is enough to find a solution of the Euler--Lagrange relations. These state that the potential of the relevant measure is constant on the support of the measure and attains its minimum there. Verifying that the potential is constant on the support of the measure is relatively easy in our situation. The difficulty is to show that outside the support the potential is at least as large as this constant. Our starting point to solve this is to express the potential as a hypergeometric function. This reduces matters to proving a convexity inequality for hypergeometric functions. In the one-dimensional case in \cite{Fr0} this was a relatively straightforward consequence of an integral representation formula for hypergeometric functions. In the present situation, however, we are outside the range of validity of this integral representation. We will show that one can obtain the desired inequality by repeated differentiation. While hypergeometric functions are a classical topic, we have not been able to find this result in the literature. We hope that our result will be relevant in other situations as well.\\
	(g) A related problem that has been receiving some recent study is that of minimizating the Riesz interaction energy with in the presence of an external field, leading to the energy functional
	\begin{equation}\label{eq:Riesz with External Field}
		\int_{\R^d} \alpha^{-1}|x|^\alpha \,d\mu(x) - \frac12 \iint_{\R^d\times\R^d} \beta^{-1} |x-y|^\beta\,d\mu(x)\,d\mu(y) \,.
	\end{equation}
	This has be studied in, for instance \cite{Bi, ChSaWo1,  ChSaWo2, HeGrBeSt, Lop, MhSa}.\\
	(h) The case where $\alpha \rightarrow \infty$ gives us the kernel
	\begin{equation}
	W_{\infty, \beta}(x-y) = \begin{cases}
	- \frac{|x-y|^{\beta}}{\beta} & \text{if}\ |x-y| \leq 1 \,, \\
	\infty & \text{if}\ |x-y| > 1 \,.
	\end{cases}
	\end{equation}
	This situation is known as hard confinement and involves a constraint on the diameter of the support of the probability measures. This topic has been studied in \cite{BrChHC, LiMC}.
\end{remarks}

Let us now put the model we are studying into context. From a very wide perspective, it is a mean-field model for the distribution of particles/individuals that interact via forces that have both attractive and repulsive components, and one would like to find a minimizing probability measure for the corresponding energy functional. Some models of this type are reviewed in \cite{Fr0}. More specifically, in the present situation the only term in the energy functional is a pair interaction, described by a certain potential $W$, and consequently the energy is of the form
$$
\frac12 \iint_{\R^d\times\R^d} W(x-y) \,d\mu(x)\,d\mu(y) \,.
$$
We refer, for instance, to the introduction of \cite{BaCaLaRa} for a large number of references of where such models appear in biology and physics. Often, there seems to be no canonical choice for the interaction potential $W$ and one direction of investigation is to understand how sensitive the qualitative behavior of minimizers is to changes in $W$. Another line of research is to understand this behavior for some toy potentials. A frequently used choice is the family $W_{\alpha,\beta}(z) = \alpha^{-1} |z|^\alpha - \beta |z|^\beta$ for $\alpha>\beta$, where one would like to understand a `phase diagram' in the $(\alpha,\beta)$-plane. This is the model studied in this paper.

Qualitative and quantitative properties of minimizers that one would like to understand are its absolute continuity with respect to Lebesgue measure \cite{CaDeMe} or the dimensionality of the support \cite{BaCaLaRa1,CaFiPa,CaSh1,DaLiMC1,DaLiMC2}. Despite some significant progress in these works, the dimensionality of minimizers seems to be still unknown in a large part of the $(\alpha,\beta)$-plane. We also emphasize that numerical experiments \cite[Table 4]{BaCaLaRa1} suggest that the dimensionality is unstable under perturbations and can decrease. The same may occur for nonradial interaction kernels $W$, as was recently explored in several works; see, e.g., \cite{MoRoSc,CaMaMoRoScVe,CaMaMoRoScVe2,CaSh2,CaSh3,MaMoRoScVe}.

We should also mention that the corresponding time-dependent equation
$$
\frac{\partial\mu}{\partial t} = \Div(\mu(\nabla W*\mu)) \,,
$$
which is formally the Wasserstein-2 gradient flow of the energy functional, has been intensely studied. For instance, in \cite{BaCaLaRa} the measures appearing in Theorem \ref{main1} were identified as steady states of the this time-dependent problem and their stability was investigated. Similarly, the functions in Theorem \ref{main} were identified as steady states in \cite{CaHu}. We recall that being a steady state means, essentially, that the potential is constant on the support of the measure. For being a global minimizer it is necessary (and, as we will see in Lemma~\ref{riesz}, also sufficient) that, apart from this constancy, the potential as at least as big as this constant outside of the support of the measure. It is this latter condition that is responsible for the technical work in this paper.

\subsection*{Acknowledgements}

The authors are grateful to D.~Bilyk, J.~A.~Carrillo, D.~Chafa\"{i}, C.~Davies, T.~Lim, J.~Mateu, R.~McCann, E.~B.~Saff, J.~Verdera, M.~Vu, and R.~Womersley for several discussions, correspondences, and help with references during the process of working on this problem.

%%%%%%%%%%%%%%%%%%%%%%%%%
%%%%%%%%%%%%%%%%%%%%%%%%%

\section{Background: Some facts about hypergeometric functions}\label{Sec:Hypergeom facts}

In the proof of our main theorems, the potentials
$$
\int_{\R^d} \left( \alpha^{-1} |x-y|^\alpha - \beta^{-1} |x-y|^\beta \right)d\mu(y) \,,
$$
of the candidate minimizers $\mu$ will play an important role. To deduce properties of these potentials, we will express them in terms of hypergeometric functions and then prove and apply results for those.

This section is meant to recall the definition and properties of hypergeometric functions and prove some that we have not been able to find in the literature. They will be crucial in the proof of our main theorems.

We restrict ourselves to parameters $a,b\in\R$ and $c\in\R\setminus(-\N_0)$ and consider the hypergeometric series \cite[(9.100)]{GrRy}
$$
F\left( a, b; c; z \right) = \sum_{n=0}^\infty \frac{a(a+1)\cdots(a+n-1)}{c(c+1)\cdots(c+n-1)}\,\frac{b(b+1)\cdots(b+n-1)}{1\cdot 2\cdots n}\, z^n \,.
$$
Here $\lambda(\lambda+1)\cdots(\lambda+n-1)$ is interpreted as $1$ if $n=0$. By standard facts about power series, this converges absolutely and uniformly on compact subsets of $\{z\in\C:\ |z|<1\}$ and defines an analytic function in the open unit disc. We will only be interested in values $z\in[0,1)$.

We first summarize some facts about its boundary behavior as $z\to 1^-$.

\begin{lemma}\label{hypercont}
	Let $a,b\in\R$, $c\in\R\setminus(-\N_0)$ with $c-a-b>0$.
	\begin{itemize}
		\item[(a)] The functions $z\mapsto F(a,b;c;z)$ on $[0,1)$ and $ z\mapsto z^{-a} F(a,b;c;z^{-1})$ on $(1, \infty)$ extend continuously to $z=1$ with
		\begin{equation}\label{eq:hyperone}
			\HG21\left( a, b; c; 1 \right)= \frac{\Gamma(c)\,\Gamma(c-a-b)}{\Gamma(c-a)\,\Gamma(c-b)} \,.
		\end{equation}
		\item[(b)] If $c-a-b>1$, then their derivatives extend continuously to $z=1$ with
		\begin{equation}\label{eq:hyperoneder}
			\frac{d}{dz}\Big|_{z=1}F\left( a, b; c; z \right) = ab\ \frac{\Gamma(c)\,\Gamma(c-a-b-1)}{\Gamma(c-a)\,\Gamma(c-b)}
		\end{equation}
		and
		\begin{equation}\label{eq:hyperonedermod}
			\frac{d}{dz}\Big|_{z=1}\left( z^{-a} F\left( a, b; c; z^{-1} \right)\right) = -a\
			\frac{\Gamma(c)\,\Gamma(c-a-b-1)}{\Gamma(c-a-1)\,\Gamma(c-b)} \,.
		\end{equation}
		\item[(c)] If $c-a-b>2$, then their second derivatives extend continuously to $z=1$ with
		\begin{equation}\label{eq:hyperoneder2}
			\frac{d^2}{dz^2}\Big|_{z=1} F\left( a, b; c; z \right) = a(a+1)b(b+1)\ \frac{\Gamma(c)\,\Gamma(c-a-b-2)}{\Gamma(c-a)\,\Gamma(c-b)}
		\end{equation}
		and
		\begin{equation}\label{eq:hyperoneder2mod}
			\frac{d^2}{dz^2}\Big|_{z=1}\left( z^{-a} F\left( a, b; c; z^{-1} \right)\right) = a(a+1)\
			\frac{\Gamma(c)\,\Gamma(c-a-b-2)}{\Gamma(c-a-2)\,\Gamma(c-b)} \,.
		\end{equation}		
	\end{itemize}
\end{lemma}

\begin{proof}
	It is well known (see, e.g., \cite[(9.102)]{GrRy}) that the hypergeometric series converges on the boundary of the unit disc if $c-a-b>0$. Its value at 1, given in \eqref{eq:hyperone}, can be found in \cite[(9.122.1)]{GrRy}. This proves part (a).
	
	Directly from the definition of the hypergeometric series we find that
	\begin{equation}
		\label{eq:hyperder}
		\frac d{dz}F\left( a, b; c; z \right) = \frac{ab}{c}\, F\left( a+1, b+1; c+1; z \right)\,.
	\end{equation}
	Therefore the facts about the derivative of the first function follow from the same facts as for the function itself. For the second function we note that
	\begin{equation}
		\label{eq:hyperdermod}
		\frac{d}{dz}\left( z^{-a} F\left( a, b; c; z^{-1} \right)\right) = -a z^{-a-1} F\left( a+1, b; c; z^{-1} \right) \,.
	\end{equation}
	This follows easily from \cite[(15.2.3)]{AbSt}. Therefore the facts about the derivative of the second function follow again from the same facts as for the function itself. This proves part~(b).
	
	Part (c) is proved in the same way.
\end{proof}

The following lemma and its corollary are the crucial facts needed for the proof of our main results. The important point is that they are proved for parameters $b$ that can be arbitrarily negative.

\begin{lemma}\label{hyperpos}
	Let $a,b\in\R$ and $c>0$ such that $c\geq\max\{a,b\}$. Then $F\left( a, b; c; z \right) \geq 0$ for all $z\in[0,1)$.
\end{lemma}

\begin{proof}
	By continuity with respect to $c$ we may assume that $c>\max\{a,b\}$. First assume that $b>0$. Then the assertion follows from the integral representation \cite[(9.111)]{GrRy}
	\begin{equation}
		\label{eq:intrephyper}
		F\left( a, b; c; z \right) = \frac{\Gamma(c)}{\Gamma(b)\,\Gamma(c-b)} \int_0^1 t^{b-1} (1-t)^{c-b-1} (1-tz)^{-a}\,dt \,,
	\end{equation}
	valid for $c>b>0$. Note that for $b>0$ we did not use the assumption $c>a$.
	
	For $b\leq 0$ we write $-\ell\geq b> -\ell-1$ for some $\ell\in\N_0$ and prove the assertion by induction on $\ell$. We can consider the case $\ell=-1$, proved above, as the base case. For the induction step, let $\ell\geq 0$. We use the formula \eqref{eq:hyperder} for the derivative of $F\left( a, b; c; z \right)$. By induction hypothesis, we have $F\left( a+1, b+1; c+1; z \right)\geq 0$. If $a\leq 0$, we have $\frac{ab}{c}\geq 0$ and we deduce $F'\left( a, b; c; z \right)\geq 0$. Thus, $F\left( a, b; c; z \right)\geq F\left( a, b; c; 0 \right)=1$ for all $z\in[0,1)$. Conversely, if $a>0$, we have $\frac{ab}{c}\leq 0$ and we deduce $F'\left( a, b; c; z \right)  \leq 0$. Thus, $F\left( a, b; c; z \right)\geq F\left( a, b; c; 1 \right)$ for all $z\in[0,1)$. The assertion therefore follows from the formula \eqref{eq:hyperone} for $F\left( a, b; c; 1 \right)$ if one notices that under our assumptions the argument of each one of the four gamma functions is positive. This completes the proof of the lemma.
\end{proof}

\begin{corollary}\label{hypercor}
	Let $a,b\in\R$ and $c>0$ with $c\geq\max\{a,b\}$. On $[0,1)$, the function $z\mapsto F\left( a, b; c; z \right)$ is
	$$
	\begin{cases}
		\text{convex} & \text{if}\ a(a+1)b(b+1)\geq 0 \,,\\
		\text{concave} & \text{if}\ a(a+1)b(b+1)\leq 0 \,.
	\end{cases}
	$$
	If, in addition, $c\geq a+2$, then on $(1,\infty)$, the function $(1,\infty)\ni z\mapsto z^{-a} F\left( a, b; c; z^{-1} \right)$ is
	$$
	\begin{cases}
		\text{convex} & \text{if}\ a(a+1)\geq 0 \,,\\
		\text{concave} & \text{if}\ a(a+1)\leq 0 \,.
	\end{cases}
	$$
\end{corollary}

\begin{proof}
	According to \eqref{eq:hyperder} and \eqref{eq:hyperdermod}, we have
	$$
	\frac{d^2}{dz^2} F\left( a, b; c; z \right) = \frac{a(a+1)b(b+1)}{c(c+1)}\, F\left( a+2, b+2; c+2; z \right)
	$$
	and
	$$
	\frac{d^2}{dz^2}\left( z^{-a} F\left( a, b; c; z^{-1} \right)\right) = a(a+1) z^{-a-2} F\left( a+2, b; c; z^{-1} \right) \,.
	$$
	Therefore, the assertion follows from Lemma \ref{hyperpos}.
\end{proof}

The above corollary, which concerns a single hypergeometric function, suffices to deal with a large part of the parameter regime in our main theorem, but in certain cases we need a more delicate result about the difference of two hypergeometric functions.

\begin{lemma}\label{lem:Hypergeom intersect once}
	If $q > 0$, $0< a_2 < a_1$, $0< b_2 < b_1$, and $c > a_1+b_1$, then the function
	$$ 
	z\mapsto F\left( a_1, b_1; c; z \right) - q F\left( a_2, b_2; c; z \right)
	$$
	has at most one zero in $[0,1]$ and, if such a zero $z_0\in[0,1]$ exists, then the function is negative in $[0,z_0)$ and positive in $(z_0,1]$.
\end{lemma}

Here, in case $z_0=0$ we use the convention $[0,z_0)=\emptyset$ and consequently we make no assertion about the function on this empty set. A similar remark applies in case $z_0=1$.

\begin{proof}
	 In terms of the Pochhammer symbol $(d)_n:=\frac{\Gamma(d+n)}{\Gamma(d)}$  the function in the lemma can be written as
	$$
	g(z) := F\left( a_1, b_1; c; z \right) - q F\left( a_2, b_2; c; z \right)  =
	\sum_{n=0}^{\infty} \Big[ \frac{(a_1)_n (b_1)_n}{(c)_n n!} - q \frac{(a_2)_n (b_2)_n}{(c)_n n!} \Big] z^n \,.
	$$
	For the proof we may assume that there is a $z_1\in[0,1]$ such that $g(z_1)\geq 0$, as otherwise there is nothing to show. We now break our proof into three steps.
	
	\medskip
	
	\emph{Step 1.} We claim that there is an $N \in \mathbb{N}_0$ such that
	$$ 
	(a_1)_n (b_1)_n \leq q(a_2)_n (b_2)_n
	\qquad\text{for}\  n< N
	$$
	and
	$$ 
	(a_1)_n (b_1)_n > q(a_2)_n (b_2)_n
	\qquad\text{for}\ n\geq N \,.
	$$
	(Here and below an assertion for $n<N$ is trivially satisfied if $N=0$.)
	
	To prove this, we first observe that there is an $n\in\N_0$ such that $(a_1)_n (b_1)_n \geq q(a_2)_n (b_2)_n$. Indeed, otherwise we have $(a_1)_n (b_1)_n < q(a_2)_n (b_2)_n$ for all $n\in\N_0$, which would imply that $g(z)<0$ for all $z\in[0,1]$, contradicting our assumption $g(z_1)\geq 0$. By choosing the minimal $n$, we obtain the existence of an $N_0\in\N_0$ such that $(a_1)_n(b_1)_n< q (a_2)_n (b_2)_n$ for all $n<N_0$ and $(a_1)_{N_0} (b_1)_{N_0} \geq q (a_2)_{N_0} (b_2)_{N_0}$.

	Our second observiation is that if $(a_1)_n (b_1)_n \geq q(a_2)_n (b_2)_n$ for some $n\in\N_0$, then
	\begin{align*}
		(a_1)_{n+1} (b_1)_{n+1} &= (a_1+n)(b_1+n) (a_1)_n (b_1)_n \\		
		& \geq (a_1 +n) (b_1 + n) q (a_2)_n (b_2)_n 
		>  q (a_2)_{n+1} (b_2)_{n+1} \,.
	\end{align*}
	This implies that $(a_1)_n (b_1)_n > q(a_2)_n (b_2)_n$ for all $n>N_0$. The claim made at the beginning of this step now follows with 
	$$
	N:=
	\begin{cases}
		N_0  & \text{if}\ (a_1)_{N_0} (b_1)_{N_0} > q (a_2)_{N_0} (b_2)_{N_0} \,,\\
		N_0 + 1 & \text{if}\ (a_1)_{N_0} (b_1)_{N_0} = q (a_2)_{N_0} (b_2)_{N_0} \,.
	\end{cases}
	$$

	\medskip
	
	\textit{Step 2.} With $N$ from Step 1 we set $p = \frac{(a_1+N)(b_1+N)}{c+N}$ and show that
	\begin{equation}
		\label{eq:differentialinequality}
		\frac1p\, g'(z) > g(z)
		\qquad\text{for all}\ z\in[0,1) \,.
	\end{equation}
	
	Indeed, by \eqref{eq:hyperder},
	\begin{align*}
		\frac{1}{p} g'(z) & = \frac{1}{p} \Big[ \frac{a_1 b_1}{c}F\left( a_1+1, b_1+1; c+1; z \right) - q \frac{a_2 b_2}{c} F\left( a_2+1, b_2+1; c+1; z \right) \Big]\\
		& = \frac{1}{p} \sum_{n=0}^{\infty} \Big[ \frac{(a_1)_{n+1} (b_1)_{n+1}}{(c)_{n+1} n!} - q \frac{(a_2)_{n+1} (b_2)_{n+1}}{(c)_{n+1} n!} \Big] z^n \\
		& = \sum_{n=0}^{\infty} \frac{(a_1 +n) (b_1 + n)}{p(c+n)} \Big[ \frac{(a_1)_{n} (b_1)_{n}}{(c)_{n} n!} - q \frac{(a_2 + n) (b_2 +n)}{(a_1 +n) (b_1 + n)} \frac{(a_2)_{n} (b_2)_{n}}{(c)_{n} n!} \Big] z^n \\
		& > \sum_{n=0}^{\infty} \frac{(a_1 +n) (b_1 + n)}{p(c+n)} \Big[ \frac{(a_1)_{n} (b_1)_{n}}{(c)_{n} n!} - q  \frac{(a_2)_{n} (b_2)_{n}}{(c)_{n} n!} \Big] z^n \,.
	\end{align*}
	In the last inequality, we used our assumptions $0<a_2<a_1$ and $0<b_2<b_1$.

	Since $\frac{(a_1 +n) (b_1 + n)}{p (c+n)}$ is an increasing function of $n$ and is one for $n = N$, we see that for $n < N$,
	$$\frac{(a_1 +n) (b_1 + n)}{p (c+n)} \Big[ \frac{(a_1)_{n} (b_1)_{n}}{(c)_{n} n!} - q  \frac{(a_2)_{n} (b_2)_{n}}{(c)_{n} n!} \Big] \geq  \frac{(a_1)_{n} (b_1)_{n}}{(c)_{n} n!} - q  \frac{(a_2)_{n} (b_2)_{n}}{(c)_{n} n!} \,,$$
	and for $n \geq N$,
	$$\frac{(a_1 +n) (b_1 + n)}{p (c+n)} \Big[ \frac{(a_1)_{n} (b_1)_{n}}{(c)_{n} n!} - q  \frac{(a_2)_{n} (b_2)_{n}}{(c)_{n} n!} \Big] > \frac{(a_1)_{n} (b_1)_{n}}{(c)_{n} n!} - q  \frac{(a_2)_{n} (b_2)_{n}}{(c)_{n} n!} \,.$$
	Inserting this into our lower bound on $g'(z)$, we obtain the differential inequality \eqref{eq:differentialinequality}.
	
	\medskip
	
	\emph{Step 3.} Using \eqref{eq:differentialinequality}, we can now complete the proof of the lemma.
	
	We first claim that $g$ has at most one zero in the half-open interval $[0,1)$ and that, if such a zero $z_0\in[0,1)$ exists, then $g<0$ on $[0,z_0)$ and $g>0$ on $(z_0,1)$. Indeed, at every zero $z_0\in[0,1)$ of $g$ we deduce from \eqref{eq:differentialinequality} that $\frac1p g'(z_0)>g(z_0)=0$. This implies that $g$ has at most one zero in $[0,1)$, for if it had more than one zero, its derivative would have to be nonpositive at at least one of the zeroes. This proves the claim about the interval $[0,1)$. (We note in passing that the same argument implies the full statement of the lemma under the additional assumption $c>a_1+b_1+1$. Indeed, under this assumption $g'$ extends continuously to $z=1$ by Lemma \ref{hypercont} and the proof above shows that \eqref{eq:differentialinequality} holds for $z=1$ as well.)
	
	We have to consider two cases, namely $g$ has either no or exactly one zero in $[0,1)$. In the latter case, we know that $g>0$ on $(z_0,1)$. By the differential inequality \eqref{eq:differentialinequality}, this implies that $g$ is increasing in $(z_0,1)$ and, consequently, $g(1)>0$. This proves the lemma in this case.
	
	Assume now that $g$ has no zero in $[0,1)$. Then either $g>0$ on $[0,1)$ or $g<0$ on $[0,1)$. In the first case we can argue as before, deducing from the differential inequality \eqref{eq:differentialinequality} that $g$ is increasing on $[0,1)$ and therefore has no zero in $[0,1]$. Meanwhile, when $g<0$ on $[0,1)$, the existence of $z_1$ with $g(z_1)\geq0$ implies that $g$ has precisely one zero in $[0,1]$ and is negative to the left of it, as claimed. This completes the proof.
\end{proof}

%%%%%%%%%%%%%%%%%%%%%%%%%
%%%%%%%%%%%%%%%%%%%%%%%%%

\section{Potentials as hypergeometric series}

Our goal in the present section is to express the potential of the candidate minimizers in our main theorems in terms of hypergeometric functions, whose definition we recalled in the previous section. We deal separately with the two parts of the potential and write our formulas in terms of a parameter $\gamma$ that will later take the values $\alpha$ or~$\beta$.

\begin{lemma}\label{pothyper2}
	Let $d\geq 2$ and $\gamma\in\R$. Then, for all $x\in\R^d$,
	\begin{align}\label{eq:pothyper2}
		&\int_{\Sph^{d-1}} |x-\omega|^\gamma \,d\omega = 2\, \pi^\frac d2 \frac{1}{\Gamma(\frac d2)} \times 
		\begin{cases}
			|x|^{\gamma} F \left(-\frac\gamma2,\frac{2-\gamma-d}{2};\frac d2;|x|^{-2} \right) & \text{if}\ |x|>1 \,,\\
			F \left(-\frac\gamma2,\frac{2-\gamma-d}{2};\frac d2;|x|^2\right) & \text{if}\ |x|<1 \,.
		\end{cases}
	\end{align}
\end{lemma}

If $\gamma>-d+1$, then Lemma \ref{hypercont} shows that the right side in \eqref{eq:pothyper2} is continuous at $|x|=1$ and the proof below shows that the identity \eqref{eq:pothyper2} extends to $|x|=1$.

\begin{proof}	
	By introducing polar coordinates, we find
	\begin{align*}
		\int_{\Sph^{d-1}} |x-\omega|^\gamma \,d\omega & = |\Sph^{d-2}| \int_0^\pi \left( |x|^2 - 2 |x|\cos\theta + 1 \right)^{\frac\gamma2} \sin^{d-2}\theta\,d\theta \\
		& = |\Sph^{d-2}| \int_{-1}^1 \left( |x|^2 - 2 |x| t + 1 \right)^{\frac\gamma2} (1-t^2)^\frac{d-3}{2} \,dt \\
		& = 2 |\Sph^{d-2}| (1+|x|)^\gamma \int_0^1 (1- \tfrac{4}{(1+|x|)^2} u)^\frac{\gamma}{2} (1-u)^{\frac{d-3}{2}} u^\frac{d-3}{2}\,du \\
		& = 2 |\Sph^{d-2}| \frac{\Gamma(\frac{d-1}{2})^2}{\Gamma(d-1)}\, (1+|x|)^\gamma \, F \left(-\tfrac\gamma2,\tfrac{d-1}{2};d-1;\tfrac{4|x|}{(1+|x|)^2}\right) \,.
	\end{align*}
	Here we have successively changed variables $\cos\theta=\omega\cdot x/|x|$, $t=\cos\theta$, and $u= (1+t)/2$ and used the integral representation of the hypergeometric series \cite[(9.111)]{GrRy}. Inserting the expression for $|\Sph^{d-2}|$, we arrive at
	\begin{equation}
		\label{eq:pothyper2alt}
		\int_{\Sph^{d-1}} |x-\omega|^\gamma \,d\omega = 2\, \pi^\frac d2 \frac{1}{\Gamma(\frac d2)} \, (|x|+1)^\gamma F \left(-\tfrac{\gamma}{2},\tfrac{d-1}{2};d-1;\tfrac{4|x|}{(|x|+1)^2}\right)
		\qquad\text{if}\ |x|\neq 1 \,.
	\end{equation}
	The claimed formula \eqref{eq:pothyper2} follows from \eqref{eq:pothyper2alt} by a transformation formula for the hypergeometric series; see \cite[(9.134.2)]{GrRy}.	
\end{proof}

Note that in the previous proof we have found an alternative expression for the left side in Lemma \ref{pothyper2}, namely \eqref{eq:pothyper2alt}. While the latter does not involve a case distinction and might be aesthetically more pleasing, we have found the formula \eqref{eq:pothyper2} the more useful one for our purposes.

\begin{lemma}\label{pothyper1}
	Let $d\geq 1$ and $-d<\gamma<-d+4$. Then, for all $x\in\R^d$,
	\begin{align}\label{eq:pothyper1}
		& \int_{|y|<1}  |x-y|^\gamma (1-|y|^2)^{\frac{2-\gamma-d}{2}}\,dy = 
		\pi^{\frac d2} \, \frac{\Gamma(\frac{4-\gamma-d}{2})\,\Gamma(\frac{\gamma+d}2)}{\Gamma(\frac d2)} \,
		 \notag \\
		& \qquad \times
		\begin{cases}
			\frac{\Gamma(\frac d2)}{\Gamma(2-\frac\gamma2) \, \Gamma(\frac{\gamma+d}{2})} \, |x|^{\gamma} F \left(-\frac\gamma2,\frac{2-\gamma-d}{2};2-\frac\gamma2;|x|^{-2}\right) & \text{if}\ |x|>1 \,,\\
			F \left(-\frac\gamma2,-1;\frac d2;|x|^2\right) & \text{if}\ |x|<1 \,.
		\end{cases}
	\end{align}
\end{lemma}

\begin{proof}
	For fixed $x\in\R^d$, both sides of the claimed formula \eqref{eq:pothyper1} are analytic functions of $\gamma$ in $\{\gamma \in \mathbb{C} :\ -d<\re\gamma<-d+4\}$, the restrictions on $\gamma$ coming on the left side of the claimed formula from the local integrability properties of the two factors in the integrand and on the right side from the domain of analyticity of the gamma functions. Consequently, it suffices to prove this formula in the range $-d < \gamma < \min\{ 0, -d+4\}$. In the remainder of this step, we impose these restrictions.
	
	In the proof, we will make repeated use of the following formula for the inverse Fourier transform of a radial function,
	\begin{equation}
		\label{eq:ftfrad}
		(2\pi)^{-\frac d2} \int_{\R^d} f(|\xi|) e^{i\xi\cdot x}\,d\xi = |x|^{-\frac{d-2}{2}} \int_0^\infty k^\frac d2 f(k) J_\frac{d-2}{2}(k|x|)\,dk \,,
	\end{equation}
	where $J_\frac{d-2}{2}$ is the Bessel function of the first kind of order $\frac{d-2}{2}$. As a first consequence of this formula, we obtain
	\begin{align}
		\label{eq:ftmin}
		(1-|x|^2)^{\frac{2-\gamma-d}{2}}\1_{(-1,1)}(|x|) = (2\pi)^{-\frac d2} \, \frac{\Gamma(\frac{4-\gamma-d}{2})}{2^\frac{\gamma+d-2}{2}} \int_{\R^d} |\xi|^{-1+\frac\gamma2} J_{1-\frac\gamma2}(|\xi|) e^{i\xi\cdot x}\,d\xi \,.
	\end{align}
Indeed, this follows from \eqref{eq:ftfrad} together with \cite[(6.575.1)]{GrRy}, which says that for $\alpha,\beta>0$ and $\re\nu+1>\re\mu>-1$,
	$$
	\int_0^\infty J_{\nu+1}(\alpha t) J_\mu(\beta t) t^{\mu-\nu} \,dt =
	\begin{cases}
		0 & \text{if}\ \alpha<\beta \,,\\
		\frac{1}{2^{\nu-\mu}\Gamma(\nu-\mu+1)}\ \frac{\beta^\mu \, (\alpha^2-\beta^2)^{\nu-\mu}}{\alpha^{\nu+1}} & \text{if}\ \alpha\geq \beta \,.
	\end{cases}
	$$

	Next, under the restrictions $-d<\gamma<0$, we have
	\begin{align}
		\label{eq:ftxalpha}
		|x|^{\gamma} = (2\pi)^{-\frac d2} \, 2^{\gamma+\frac d2} \,\frac{\Gamma(\frac{\gamma+d}2)}{\Gamma(-\frac{\gamma}2)}\, \int_{\R^d} |\xi|^{-d-\gamma} e^{i\xi\cdot x}\,d\xi \,.
	\end{align}
	This is formula is well known. It can also be derived from \cite[(6.561.14)]{GrRy} via \eqref{eq:ftfrad}.
	
	Combining \eqref{eq:ftmin} and \eqref{eq:ftxalpha}, we find that
	$$
	\int_{|y|<1} |x-y|^\gamma (1-|y|^2)^{\frac{2-\gamma-d}{2}}\,dy = 2^{\frac\gamma 2+1} \, \frac{\Gamma(\frac{4-\gamma-d}{2})\,\Gamma(\frac{\gamma+d}2)}{\Gamma(-\frac{\gamma}2)} \int_{\R^d} |\xi|^{-d-1-\frac\gamma2} J_{1-\frac\gamma2}(|\xi|) e^{i\xi\cdot x}\,d\xi \,.
	$$
	Using once again \eqref{eq:ftfrad}, we can rewrite the right side as
	$$
	(2\pi)^{\frac d2} \, 2^{\frac\gamma 2+1} \frac{\Gamma(\frac{4-\gamma-d}{2})\,\Gamma(\frac{\gamma+d}2)}{\Gamma(-\frac{\gamma}2)} |x|^{-\frac{d-2}{2}} \int_0^\infty k^{-\frac{\gamma+d+2}2} J_{1-\frac\gamma 2}(k) J_{\frac{d-2}2}(k|x|)\,dk \,.
	$$
	The formula in the lemma now follows from formulas \cite[(6.574.1) \& (6.574.3)]{GrRy}, which say that
	$$
	\int_0^\infty \! J_\nu(\alpha t) J_\mu(\beta t) t^{-\lambda}\,dt \! = \!
	\begin{cases}
		& \!\!\! \frac{\Gamma(\frac{\nu+\mu-\lambda+1}{2})}{2^\lambda\,\Gamma(\frac{-\nu+\mu+\lambda+1}{2})\,\Gamma(\nu+1)}
		\frac{\alpha^\nu}{\beta^{\nu-\lambda+1}} F \left( \frac{\nu+\mu-\lambda+1}{2},\frac{\nu-\mu-\lambda+1}{2};\nu+1;\frac{\alpha^2}{\beta^2}\right) \\
		& \qquad \text{if}\ 0<\alpha<\beta \,,\\
		& \\
		& \!\!\! \frac{\Gamma(\frac{\nu+\mu-\lambda+1}{2})}{2^\lambda\,\Gamma(\frac{\nu-\mu+\lambda+1}{2})\,\Gamma(\mu+1)}
		\frac{\beta^\mu}{\alpha^{\mu-\lambda+1}} F \left(\frac{\nu+\mu-\lambda+1}{2},\frac{-\nu+\mu-\lambda+1}{2};\mu+1;\frac{\beta^2}{\alpha^2}\right) \\
		& \qquad \text{if}\ 0<\beta<\alpha \,,
	\end{cases}
	$$
	provided that $\re(\nu+\mu-\lambda+1)>0$ and $\re\lambda>-1$. (We have corrected two misprints in \cite{GrRy} in the second line of the formula. Namely, we replaced $\beta^\nu$ by $\beta^\mu$ and we replaced $\Gamma(\nu+1)$ by $\Gamma(\mu+1)$. Indeed, the second line of the formula follows from the first one if one interchanges simultaneously the roles of $\alpha$ and $\beta$ and of $\nu$ and $\mu$. Note also that the restriction $\re(\nu+\mu-\lambda+1)>0$ is satisfied in our case since $\gamma<0$. This completes the proof of Lemma \ref{pothyper1}.	
\end{proof}

We end this section by noting that Lemma \ref{pothyper2} can alternatively be proved similarly as Lemma~\ref{pothyper1}.

\begin{proof}[Second proof of Lemma \ref{pothyper2}]
	We only sketch the differences to the proof of Lemma \ref{pothyper1}. Again, using analyticity, we restrict ourselves to the range $-d<\gamma<0$.
	
	The role of formula \eqref{eq:ftmin} is played by the formula
	$$
	\delta_{\Sph^{d-1}}(x) = (2\pi)^{-\frac d2} \int_{\R^d} |\xi|^{-\frac{d-2}{2}} J_\frac{d-2}{2}(|\xi|) e^{i\xi\cdot x}\,d\xi \,.
	$$
	This follows by Fourier inversion from the formula \eqref{eq:ftfrad} when $f$ is a delta measure at $k=1$.
	
	Combining this with \eqref{eq:ftxalpha} and \eqref{eq:ftfrad}, we obtain
	$$
	\int_{\Sph^{d-1}} |x-\omega|^\gamma \,d\omega = (2\pi)^\frac d2 2^{\gamma+\frac d2} \frac{\Gamma(\frac{\gamma+d}{2})}{\Gamma(\frac\gamma 2)}\, |x|^{-\frac{d-2}{2}} \int_0^\infty k^{-d-\gamma+1} J_\frac{d-2}{2}(k) J_\frac{d-2}{2}(k|x|)\,dk \,.
	$$
	The claimed formula now follows using the formula for the integral of two Bessel functions with a power function, given in the previous proof.
\end{proof}

If we use \cite[(6.576.2)]{GrRy} for the integral of the product of two Bessel functions instead of the formula given in the proof of Lemma \ref{pothyper1}, we arrive at \eqref{eq:pothyper2alt}.
	
%%%%%%%%%%%%%%%%%%%%%%%%%
%%%%%%%%%%%%%%%%%%%%%%%%%

\section{Convexity}

The Euler--Lagrange equations for our minimization problem say that the potential is constant on the support of the minimizing measure and at least as big outside of this support; see, e.g., \cite{BaCaLaRa1}. The following lemma, which says that these necessary conditions for a minimizer are, in fact, sufficient, plays a fundamental role in the proof of our main results. It is at this point that the assumption $2\leq\alpha\leq 4$ enters.

\begin{lemma}\label{riesz}
	Let $d\geq 1$ and let $-d<\beta\leq 2 \leq \alpha\leq 4$ with $\beta<\alpha$. Assume that there are $\mu\in P(\R^d)$ and $\eta\in\R$ such that
	$$
	\phi_{\alpha,\beta}(x) := \int_{\R^d} \left( \alpha^{-1} |x-y|^\alpha - \beta^{-1} |x-y|^\beta \right)d\mu(y) \,,
	\qquad x\in\R^d \,,
	$$
	satisfies
	\begin{equation*}
		\phi_{\alpha,\beta}\geq \eta
		\quad\text{on}\ \R^d
		\qquad\text{and}\qquad
		\phi_{\alpha,\beta} = \eta
		\quad\text{on}\ \supp\mu \,.
	\end{equation*}
	Then $\mu$ is a minimizer for $\mathcal{E}_{\alpha,\beta}$ and $\eta=2E_{\alpha,\beta}$.
	If $(\alpha,\beta)\neq (4,2)$, then $\mu$ is the unique minimizer up to translations.
\end{lemma}

We only sketch the proof, as the details of the argument are similar as in the proof of \cite[Lemma 6]{Fr1}. The fundamental observation is the fact that the Fourier transform of $-\beta^{-1}|z|^{\beta}$ is positive for $-d<\beta<0$ and that its restriction to $\R^d\setminus\{0\}$ is positive for $0\leq\beta<2$. The behavior at the point $0$ is irrelevant since the object that is Fourier transformed is the difference of two probability measures, which has integral zero, and therefore Fourier transform vanishing at the origin. Moreover, we use the fact that the restriction to $\R^d\setminus\{0\}$ of the Fourier transform of $\alpha^{-1}|z|^\alpha$ is positive. As observed by Lopes \cite{Lo}, the behavior at the point $0$ is irrelevant since the argument is applied to the difference of two probability measures with equal center of mass; see \cite[Theorem 27]{CaDeDoFrHo} for an application of this argument. When $(\alpha,\beta)\neq (4,2)$, then at least one of the convexities is strict and we obtain the uniqueness (up to translations) of the minimizer.

%%%%%%%%%%%%%%%%%%%%%%%%%
%%%%%%%%%%%%%%%%%%%%%%%%%

\section{Proof of Theorems \ref{main1} and \ref{main}}

In this section we prove the main results stated in the introduction. We rely on Lemma \ref{riesz} and so only need to verify the conditions stated there are satisfied by the potential of our candidate minimizing measure. In the setting of Theorem \ref{main1} the condition that the potential is constant on the support of the candidate measure is trivially satisfied by radial symmetry. The nontrivial part of the proof is to show that potential is everywhere at least as big as on the support of the measure. This is the main technical work, which relies on our lemmas about hypergeometric functions.

We now present the details.

\begin{proof}[Proof of Theorem \ref{main1}.]
	We assume $d\geq2$, $2\leq\alpha\leq 4$, and $\beta_*(\alpha)\leq\beta\leq 2$ with $\alpha>\beta$. We only prove the theorem for $\beta\neq 0$, the case $\beta=0$ being similar; see Section \ref{Sec:Log Case} for more details. We break our proof into two steps: finding the radius of the sphere, then showing that the uniform measure on that sphere is indeed a minimizer.
	
	\medskip
	
	\emph{Step 1.} It is convenient to introduce the function
	\begin{equation}
		\label{eq:defpsigamma}
		\psi_\gamma(\rho) := 
		\begin{cases}
			\rho^{\frac\gamma2} F \left(-\frac\gamma2,\frac{2-\gamma-d}{2};\frac d2;\rho^{-1} \right) & \text{if}\ \rho>1 \,,\\
			F \left(-\frac\gamma2,\frac{2-\gamma-d}{2};\frac d2;\rho\right) & \text{if}\ \rho<1 \,.
		\end{cases}
	\end{equation}
	Under the assumption $d+\gamma>2$, it follows from Lemma \ref{hypercont}, together with some manipulations of gamma functions, that $\psi_\gamma$, which is originally defined only $[0,\infty)\setminus\{1\}$, extends to a continuously differentiable function on $[0,\infty)$. Moreover, Lemma \ref{hypercont} shows that
	\begin{equation}
		\label{eq:psivalues2}
		\psi_\gamma(1) = \frac{\Gamma(\frac d2)\,\Gamma(d+\gamma-1)}{\Gamma(\frac{d+\gamma}2)\,\Gamma(\frac{2d+\gamma-2}{2})}
		\qquad\text{and}\qquad
		\psi_\gamma'(1) = \frac\gamma2\, \frac{\Gamma(\frac d2)\,\Gamma(d+\gamma-2)}{\Gamma(\frac{d+\gamma-2}2)\,\Gamma(\frac{2d+\gamma-2}{2})} \,.
	\end{equation}
	We will apply this with $\gamma$ equal to $\alpha$ and $\beta$, for which the assumption $d+\gamma>2$ is satisfied.
	
	The functions $\psi_\gamma$ are relevant, since by Lemma \ref{pothyper2} and scaling, we can express the total potential of the measure $(|\Sph^{d-1}| R^{d-1})^{-1} \delta_{\partial B_R(0)}$ as
	\begin{align}\label{eq:totpotmain}
		\frac{1}{|\Sph^{d-1}|}\int_{\Sph^{d-1}} \left( \frac{|x-R\omega|^\alpha}{\alpha}  - \frac{|x-R\omega|^\beta}{\beta} \right) d\omega = \frac{R^{\alpha}}{\alpha}  \psi_\alpha \Big(\Big|\frac{x}{R} \Big|^2 \Big) -\frac{ R^{\beta}}{\beta} \psi_\beta \Big(\Big|\frac{x}{R} \Big|^2\Big) \,.
	\end{align}
	In particular, this total potential is radially symmetric and therefore constant on spheres. We want to choose the radius $R$ of the sphere in such a way that the total potential is minimal at $|x|=R$. The differentiability of $\psi_\gamma$ implies that the total potential is continuously differentiable with respect to $|x|$. Setting its derivative equal to zero, we arrive at the condition
	\begin{equation}
		\label{eq:ralpha2}
		\alpha^{-1} R^\alpha \psi_\alpha'(1) - \beta^{-1} R^\beta \psi_\beta'(1) = 0 \,.
	\end{equation}
	In view of the expression for $\psi_\gamma'(1)$ in \eqref{eq:psivalues2}, we see that this is satisfied if (and only if)
	$$
	R= \left( \frac{\alpha}{\beta} \frac{\psi_\beta'(1)}{\psi_\alpha'(1)} \right)^\frac1{\alpha-\beta}
	= \left( \frac{\Gamma(d+\beta-2)}{\Gamma(d+\alpha-2)} \, 
	\frac{\Gamma(\frac{d+\alpha-2}{2})\,\Gamma(\frac{2d+\alpha-2}{2})}{\Gamma(\frac{d+\beta-2}{2})\,\Gamma(\frac{2d+\beta-2}{2})} \right)^\frac{1}{\alpha-\beta} = R_{\alpha,\beta} \,.
	$$
	For the last equality we recall that $R_{\alpha,\beta}$ was defined in \eqref{eq:defr1}, and use the duplication formula for the gamma function.
	
	According to \eqref{eq:totpotmain}, the value of the total potential at $|x|=R_{\alpha,\beta}$ is equal to
	\begin{align*}
		|\Sph^{d-1}|^{-1} \int_{\Sph^{d-1}} \left( \alpha^{-1} |x-R_{\alpha,\beta}\omega|^\alpha - \beta^{-1}|x-R_{\alpha,\beta} \omega|^\beta \right) d\omega \Big|_{|x|=R_{\alpha,\beta}} = \eta
	\end{align*}
	with
	\begin{align*}
		\eta & := \alpha^{-1} R_{\alpha,\beta}^{\alpha} \psi_\alpha(1) - \beta^{-1} R_{\alpha,\beta}^{\beta} \psi_\beta(1)\\
		& \; = \left( \alpha^{-1} \psi_\alpha(1) - \beta^{-1} \psi_\beta(1) R_{\alpha,\beta}^{-\alpha+\beta} \right) R_{\alpha,\beta}^\alpha \\
		&\; = \alpha^{-1} \left( \psi_\alpha(1) - \psi_\beta(1) \frac{\psi_\alpha'(1)}{\psi_\beta'(1)} \right) R_{\alpha,\beta}^\alpha \\
		& \; = - \pi^{-\frac12} \, 2^{d+\alpha-2}\, \frac{\Gamma(\frac d2)\,\Gamma(\frac{d+\alpha-1}{2})}{\Gamma(\frac{2d+\alpha-2}{2})} \left( \frac1\beta - \frac1\alpha \right) R_{\alpha,\beta}^\alpha \,.
	\end{align*}
	Here we used \eqref{eq:ralpha2}, \eqref{eq:psivalues2}, and the duplication formula for the gamma function.
	
	To summarize our discussion so far, we have chosen the radius $R=R_{\alpha,\beta}$ in such a way that the potential has a critical point at $|x|=R_{\alpha,\beta}$. It still remains to be shown that the potential has a global minimum at $|x|=R_{\alpha,\beta}$. Once we have shown this, we can apply Lemma \ref{riesz} and deduce that $(|\Sph^{d-1}| R_{\alpha,\beta}^{d-1})^{-1}\delta_{R{_\alpha,\beta} \Sph^{d-1}}$ is a minimizer and we obtain the value $E_{\alpha,\beta}=\eta/2$ for the minimal energy, which coincides with the value stated in Theorem \ref{main}. If $(\alpha,\beta)\neq(4,2)$, then the lemma also implies that the minimizer is unique up to translations.
	
	\medskip
	
	\emph{Step 2.} 	
	In view of \eqref{eq:totpotmain}, the total potential having a global minimum at $|x|=R_{\alpha,\beta}$ is equivalent to the inequality
	\begin{align}\label{eq:ineqmain1}
		\frac{ R_{\alpha,\beta}^{\alpha}}{\alpha} \psi_\alpha \Big( \Big|\frac{x}{R_{\alpha,\beta}} \Big|^2 \Big) - \frac{ R_{\alpha,\beta}^{\beta}}{\beta} \psi_\beta\Big( \Big|\frac{x}{R_{\alpha,\beta}} \Big|^2 \Big) \geq \frac{ R_{\alpha,\beta}^{\alpha}}{\alpha} \psi_\alpha(1) - \frac{ R_{\alpha,\beta}^{\beta}}{\beta} \psi_\beta(1) \qquad\text{for all}\ x\in\R^d \,.
	\end{align}
	This can be simplified to
	$$
	\alpha^{-1} R_{\alpha,\beta}^{\alpha-\beta} \psi_\alpha(\rho) - \beta^{-1} \psi_\beta(\rho) \geq  \alpha^{-1} R_{\alpha,\beta}^{\alpha-\beta} \psi_\alpha(1) - \beta^{-1} \psi_\beta(1)
	\qquad\text{for all}\ \rho>0 \,.
	$$	
	Using \eqref{eq:ralpha2} to relate $R_{\alpha,\beta}$ and $\psi_\beta'(1)/\psi_\alpha'(1)$, it becomes
	$$
	\beta^{-1} \frac{\psi_\beta'(1)}{\psi_\alpha'(1)}\, \psi_\alpha(\rho) - \beta^{-1} \psi_\beta(\rho) \geq \beta^{-1} \frac{\psi_\beta'(1)}{\psi_\alpha'(1)}\, \psi_\alpha(1) - \beta^{-1} \psi_\beta(1)
	\qquad\text{for all}\ \rho>0 \,.
	$$	
	To prove this inequality, it suffices to show that the function 
	\begin{equation}
		\label{eq:goalconvex}
		\rho\mapsto \beta^{-1} \frac{\psi_\beta'(1)}{\psi_\alpha'(1)}\, \psi_\alpha(\rho) - \beta^{-1} \psi_\beta(\rho)
	\end{equation}
	is convex on $[0,\infty)$. This is the content of Proposition \ref{psiconvex} in the next section. Since the function is differentiable at $\rho=1$ and has vanishing derivative there, this will prove the desired inequality and complete the proof of Theorem \ref{main1}.
\end{proof}

We now proceed to the proof of Theorem \ref{main}. The overall strategy is the same as for Theorem \ref{main1}, and we briefly outline it before proceeding. We note that on the support of the candidate minimizer,  the $\beta$-part of the potential (given by \eqref{eq:pothyper1} with $\gamma=\beta$)  is a hypergeometric function whose second index is $-1$, which means that the hypergeometric function is an affine-linear function. In other words, the $\beta$-potential inside the support is of the form $b_1 + b_2 |x|^2$. We then show that the quadratic, i.e. $\alpha=2$, part of the potential is of the same form. By an appropriate scaling, one can arrange that the coefficients of $|x|^2$ cancel each other, so that the total potential is constant on the support of the candidate minimizer. It remains to show that the potential is bounded from below by this constant. This is again the main technical work, which we deduce from our lemmas about hypergeometric functions.

\begin{proof}[Proof of Theorem \ref{main}]
	 We will prove the theorem for $\beta\neq 0$, the case $\beta=0$ being similar; see Section \ref{Sec:Log Case} for details. Recall that $\alpha = 2$ and $-d<\beta<\min\{-d+4,2\}$.
	
	We first need an expression for the total potential. We denote the $\beta$-part of the potential by
	\begin{equation}
		\label{eq:defphimain}
		\Phi(x) = - \beta^{-1} \int_{|y|<1} |x-y|^\beta (1-|y|^2)^{\frac{2-\beta-d}{2}}\,dy
	\end{equation}
	and recall that we have obtained an expression for this in Lemma \ref{pothyper1}. To compute the $\alpha$-part of the potential, we recall the definition \eqref{eq:defca} of $C_\beta$. Expressing a beta function in terms of gamma functions, we find
	\begin{align} \label{eq:defcacomp}
		\int_{|y|<1} (1-|y|^2)^\frac{2-\beta-d}{2}\,dy & = |\Sph^{d-1}| \int_0^1 (1-r^2)^{\frac{2-\beta-d}{2}} r^{d-1}\,dr  \notag \\
		&  = 2^{-1} |\Sph^{d-1}| \int_0^1 (1-t)^{\frac{2-\beta-d}{2}} t^{\frac{d-2}2}\,dt \notag \\
		& = 2^{-1} |\Sph^{d-1}| \, \frac{\Gamma(\frac d2)\,\Gamma(\frac{4-\beta-d}{2})}{\Gamma(\frac{4-\beta}{2})} = \pi^\frac d2\, \frac{\Gamma(\frac{4-\beta-d}{2})}{\Gamma(\frac{4-\beta}{2})}  = C_\beta 
	\end{align}
	and, similarly,
	$$
	\int_{|y|<1} |y|^2 (1-|y|^2)^\frac{2-\beta-d}{2}\,dy = \pi^\frac d2\, \frac d2\ \frac{\Gamma(\frac{4-\beta-d}{2})}{\Gamma(\frac{6-\beta}{2})} = \frac{ d}{4-\beta} \, C_\beta \,.
	$$
	Thus,
	\begin{equation}\label{eq:Quadratic part of potential}
	2^{-1} \int_{|y|<1} |x-y|^2(1-|y|^2)^{\frac{2-\beta-d}{2}}\,dy = 2^{-1} C_\beta \left( |x|^2 + \frac{d}{4-\beta} \right).
	\end{equation}
	
	It follows by scaling that
	\begin{equation*}
	\int_{|y|<R}\!\! \left(\frac{|x-y|^2}{2}  - \frac{|x-y|^\beta}{\beta} \right) (R^2-|y|^2)^\frac{2-\beta-d}{2}dy  = R^2 \Phi\Big(\frac{x}{R}\Big) + \frac{ C_\beta}{2} R^{4-\beta} \left( \Big|\frac{x}{R}\Big|^2 \! + \! \frac{d}{4-\beta} \right)\!.
	\end{equation*}
	Since $ F \left(a,-1;c;z\right)= 1- (a/c) z$, it follows from Lemma \ref{pothyper1} that
	\begin{align*}
		\Phi(x) & = - \beta^{-1} \pi^{\frac d2} \, \frac{\Gamma(\frac{4-\beta-d}{2})\,\Gamma(\frac{\beta+d}2)}{\Gamma(\frac d2)} \left( 1 + \frac{\beta}{d}\,|x|^2\right) \\
		& = - 2^{-1} C_\beta R_{2,\beta}^{2-\beta} \left( \frac{d}{\beta} + |x|^2 \right)
		\qquad\qquad\qquad\text{if}\ |x|<1 \,.
	\end{align*}
	Here we have used the definition of $R_{2,\beta}$ from \eqref{eq:defr2}. We now see that with the choice $R=R_{2,\beta}$ the coefficient in front of $|x|^2$ vanishes and we have
	\begin{align*}
		& \int_{|y|<R_{2,\beta}} \left( \frac{|x-y|^2}2 - \frac{|x-y|^\beta}\beta \right) (R_{2,\beta}^2-|y|^2)^\frac{2-\beta-d}{2}\,dy \\
		& = - 2^{-1} C_\beta R_{2,\beta}^{4-\beta} \frac{d}{\beta} + 2^{-1} \frac{d}{4-\beta}  C_\beta R_{2,\beta}^{4-\beta} 
		= - C_\beta R_{2,\beta}^{4-\beta} \frac{d(2-\beta)}{\beta(4-\beta)}
		\qquad\text{if}\ |x|<R_{2,\beta} \,.
	\end{align*}
	Therefore $d\mu(x) = C_\beta^{-1} R_{2,\beta}^{-2+\beta} (R_{2,\beta}^2-|x|^2)^\frac{2-\beta-d}{2} \mathbbm{1}_{B_{R_{2, \beta}}}(x)\,dx$ is a probability measure that satisfies the second condition in Lemma \ref{riesz} with
	$$
	\eta = - R_{2,\beta}^2 \frac{d(2-\beta)}{\beta(4-\beta)} \,.
	$$
	The task is now to show that the first condition in the lemma is satisfied as well, namely that the total potential of $\mu$ is at least $\eta$ for $|x|\geq R_{2,\beta}$. Once this is shown, we infer from the lemma that $\mu$ is the unique (up to translations) minimizer and we obtain the value for the minimal energy stated in Theorem \ref{main}.
	
	At this point it is convenient to introduce the function
	$$
	\psi(\rho) := 
	\begin{cases}
		\frac{\Gamma(\frac d2)}{\Gamma(2-\frac\beta2) \, \Gamma(\frac{\beta+d}{2})} \, \rho^{\frac\beta2}  F \left(-\frac\beta2,\frac{2-\beta-d}{2};2-\frac\beta2;\rho^{-1} \right) & \text{if}\ \rho>1 \,,\\
		  F \left(-\frac\beta2,-1;\frac d2;\rho \right) & \text{if}\ \rho<1 \,.
	\end{cases}
	$$
	Using \eqref{eq:pothyper1}, we may rewrite $\Phi$ in terms of $\psi$:
	$$
	\Phi(x) = -\beta^{-1} \pi^{\frac d2} \, \frac{\Gamma(\frac{4-\beta-d}{2})\,\Gamma(\frac{\beta+d}2)}{\Gamma(\frac d2)} \, \psi(|x|^2) \,.
	$$
	Using some straightforward manipulations with gamma functions, we deduce from Lemma \ref{hypercont} that $\psi$, which is originally defined only in $[0,\infty)\setminus\{1\}$, extends to a continuously differentiable function on $[0,\infty)$. This continuous differentiability, together with the fact that $\psi$ is affine-linear in $(0,1)$, implies that $ F \left(-\frac\beta2,-1;\frac d2;\rho \right) = \psi(1) + \psi'(1)(\rho-1)$. This allows us to write, for all $x\in\R^d$,
	\begin{align*}
		& \int_{|y|<R_{2,\beta}} \left( \frac{|x-y|^2}2 - \frac{|x-y|^\beta}\beta \right) (R_{2,\beta}^2-|y|^2)^\frac{2-\beta-d}{2}\,dy \\
		& = - C_\beta R_{2,\beta}^{4-\beta} \frac{d(2-\beta)}{\beta(4-\beta)} \\
		& \quad - \beta^{-1} \pi^{\frac d2} \, \frac{\Gamma(\frac{4-\beta-d}{2})\,\Gamma(\frac{\beta+d}2)}{\Gamma(\frac d2)} 
		\left( \psi \Big( \Big| \frac x{R_{2,\beta}}\Big|^2\Big) - \psi(1) - \psi'(1) \Big(\Big| \frac x{R_{2,\beta}}\Big|^2 - 1 \Big) \right).
	\end{align*}
	The desired inequality for the total potential is therefore equivalent to the inequality
	$$
	-\beta^{-1} \left( \psi(\rho) - \psi(1) - \psi'(1)(\rho-1) \right) \geq 0
	\qquad\text{for all}\ \rho\geq 1 \,.
	$$
	This inequality is a consequence of the convexity of $\psi$ for $\beta<0$ and its concavity for $\beta\in(0,1)$, which we have shown in Corollary \ref{hypercor}. This completes the proof of Theorem~\ref{main}.
\end{proof}

%%%%%%%%%%%%%%%%%%%%%%%%%%%%%%%%%%
%%%%%%%%%%%%%%%%%%%%%%%%%%%%%%%%%%

\section{Completion of the proof of Theorem \ref{main1}}

The proof of Theorem \ref{main1} in the previous section relied on the convexity of a certain function, which we will justify in this section.

Throughout this section, we assume that $d\geq 2$, $2\leq\alpha\leq 4$, $\beta_*(\alpha)\leq\beta\leq 2$ and $\beta<\alpha$. We introduce the function 
\begin{equation}
	\label{eq:goalconvexapp}
	\Psi(\rho) := \beta^{-1} \frac{\psi_\beta'(1)}{\psi_\alpha'(1)}\, \psi_\alpha(\rho) - \beta^{-1} \psi_\beta(\rho) \,,
\end{equation}
where $\psi_\alpha$ and $\psi_\beta$ are defined by \eqref{eq:defpsigamma}. To lighten the notation, we do not reflect the dependence of $\Psi$ on $d$, $\alpha$ and $\beta$ in the notation. The main result of this section is the following proposition, which completes the proof of Theorem \ref{main1}.

\begin{proposition}\label{psiconvex}
	The function $\Psi$ is convex on $[0,\infty)$.
\end{proposition}

For the proof of this proposition we will have to distinguish various cases. We begin with the easiest one.

\begin{proof}[Proof of Proposition \ref{psiconvex} for $\beta\geq-d+4$]
	It follows from Corollary \ref{hypercor} that $\beta^{-1}\psi_\beta$ is concave and $\psi_\alpha$ is convex on each one of the two intervals $(0,1)$ and $(1,\infty)$ (note that the inequality $\beta\geq-d+4$ guarantees the assumption $c\geq a+2$ of the corollary.) Since the derivatives of both functions extend continuously to the point $\rho=1$, we deduce that both functions have the respective concavity and convexity properties on all of $(0,\infty)$. Moreover, since, according to \eqref{eq:psivalues2}, $\beta^{-1}\psi_\beta'(1)\geq 0$ and $\psi_\alpha'(1)>0$, we obtain the claimed convexity of the function \eqref{eq:goalconvex}. This concludes the proof of the proposition for $\beta\geq -d+4$.
\end{proof}

In the remainder of this section, we will prove Proposition \ref{psiconvex} for $\beta<-d+4$. As in the preceding proof, we will show separately the convexity on $[0,1]$ and on $[1,\infty)$. This is the respective content of the following two lemmas.

\begin{lemma}\label{convexsmall}
	Let $d\geq 2$, $2\leq\alpha\leq 4$ and $\beta_*(\alpha)\leq\beta<-d+4$. Then $\Psi$ is convex on $[0,1]$.
\end{lemma}

\begin{proof}
	\emph{Step 1.} We show that $\Psi$ is twice continuously differentiable on $[0,\infty)$ and that $\Psi''(1)\geq 0$ (we note that the twice differentiability on $(1,\infty)$ and the continuity of the second derivative at $1$ will not be needed in our argument.)
		
	Under the assumption $d+\gamma>3$, it follows from Lemma \ref{hypercont}, together with some manipulations of gamma functions, that $\psi_\gamma$ is twice continuously differentiable on $[0,\infty)$ and that
\begin{equation}
	\label{eq:psivalues3}
	\psi_\gamma''(1) = \frac\gamma2(\frac\gamma2-1) \, \frac{\Gamma(\frac d2)\,\Gamma(d+\gamma-3)}{\Gamma(\frac{d+\gamma-4}2)\,\Gamma(\frac{2d+\gamma-2}{2})} \,.
\end{equation}
Since $\beta_*(\alpha)>-d+3$ for $d\geq 2$ and $2\leq\alpha\leq 4$, the assumption $d+\gamma>3$ in the lemma is satisfied for both $\gamma=\alpha$ and $\gamma=\beta$. It follows that $\Psi$ is twice continuously differentiable on $[0,\infty)$ with
\begin{align*}
	\Psi''(1) & = \beta^{-1} \frac{\psi_\beta'(1)}{\psi_\alpha'(1)}\, \psi_\alpha''(1) - \beta^{-1} \psi_\beta''(1) \\
	& = \frac12 (\frac\alpha2-1) \,
	\frac{ \Gamma(d+\beta-2)\, \Gamma(\frac{d+\alpha-2}2)} 
	{ \Gamma(d+\alpha-2)\,\Gamma(\frac{d+\beta-2}2)\,\Gamma(\frac{2d+\beta-2}{2})}
	\, \frac{\Gamma(\frac d2)\,\Gamma(d+\alpha-3)}{\Gamma(\frac{d+\alpha-4}2)} \\
	& \quad 
	- \frac12 (\frac\beta2-1) \, \frac{\Gamma(\frac d2)\,\Gamma(d+\beta-3)}{\Gamma(\frac{d+\beta-4}2)\,\Gamma(\frac{2d+\beta-2}{2})} \\
	& = \frac12 \frac{\Gamma(\frac d2)\,\Gamma(d+\beta-3)}{\Gamma(\frac{d+\beta-4}{2})\,\Gamma(\frac{2d+\beta-2}{2})}
	\left(
	(\frac\alpha2-1) \frac{(d+\beta-3)\, \frac{d+\alpha-4}{2}}{(d+\alpha-3)\, \frac{d+\beta-4}{2}}
	- (\frac\beta2-1) \right).
\end{align*}
Since $-d+3<\beta<-d+4$, we have
$$
\frac12 \frac{\Gamma(\frac d2)\,\Gamma(d+\beta-3)}{\Gamma(\frac{d+\beta-4}{2})\,\Gamma(\frac{2d+\beta-2}{2})}<0 \,.
$$
Meanwhile, the condition $\beta\geq\beta_*(\alpha) = \frac{-10+3\alpha+7d+\alpha d-d^2}{d+\alpha-3}$ means
$$
(\frac\alpha2-1) \frac{(d+\beta-3)\, \frac{d+\alpha-4}{2}}{(d+\alpha-3)\, \frac{d+\beta-4}{2}}
- (\frac\beta2-1) \leq 0 \,.
$$
This can be seen by setting the expression on the left side equal to zero and solving the resulting quadratic equation for $\beta$. The number of arithmetic manipulations can be reduced by noting that $\beta=\alpha$ is one of the solutions of the quadratic equation. Factoring out $\beta-\alpha$ one arrives at a first order equation for $\beta$. Thus, at the end, this argument shows that $\Psi''(1)\geq 0$, as claimed.

\medskip

\emph{Step 2.} We now assume that $\alpha=4$ if $d=2$, $3\leq\alpha\leq 4$ if $d=3$ and $2\leq\alpha\leq 4$ if $d\geq 4$ and show that $\Psi'''\leq 0$ on $(0,1)$.

Note that together with the information $\Psi''(1)\geq 0$ from Step 1, this implies that $\Psi''\geq 0$ on $[0,1]$. This proves the lemma under the above extra assumptions on $\alpha$.

	To prove the nonpositivity of
	$$
	\Psi''' = \beta^{-1} \frac{\psi_\beta'(1)}{\psi_\alpha'(1)}\, \psi_\alpha''' - \beta^{-1} \psi_\beta'''
	$$
	on $(0,1)$, recall from \eqref{eq:psivalues2} that $\beta^{-1}\psi_\beta'(1)\geq 0$ and $\psi_\alpha'(1)>0$. Thus, it will be sufficient to show that $\psi_\alpha'$ is concave on $(0,1)$ and that $\beta^{-1} \psi_\beta'$ is convex on $(0,1)$.
	
	According to \eqref{eq:hyperder} and \eqref{eq:hyperdermod} we have
	$$
	\psi_\gamma'(\rho) = 
	\begin{cases}
		(-\frac\gamma2) \frac{2-\gamma-d}{d}\,  F (-\frac\gamma2+1,\frac{4-\gamma-d}{2};\frac d2+1;\rho) & \text{if}\ \rho<1 \,,\\
		\frac\gamma2\, \rho^{\frac{\gamma-2}{2}}    F(-\frac\gamma2+1,\frac{2-\gamma-d}{2};\frac d2;\rho^{-1}) & \text{if}\ \rho> 1\,.		
	\end{cases}
	$$
	The claimed concavity and convexity properties of $\psi_\alpha'$ and $\beta^{-1}\psi_\beta'$ now follow from Corollary \ref{hypercor}. The additional assumptions on $\alpha$ arise as follows: When considering $\gamma=\alpha$ we apply the corollary with $a=\frac{-\alpha}2+1$, $b=\frac{4-\alpha-d}{2}$, $c=\frac d2+1$. Since the prefactor $-\frac\alpha2 \frac{2-\alpha-d}{d}$ is positive, we need to verify that $a(a+1)b(b+1)\leq 0$. Since $a(a+1) = (-\frac\alpha2+1)(-\frac\alpha2+2)$ is nonpositive, we need to verify that $b(b+1) = \frac{4-\alpha-d}{2}\frac{6-\alpha-d}{2}$ is nonnegative. Recall that $\alpha$ ranges over $[2,4]$. When $d\geq 4$, then $(6-\alpha-d)(4-\alpha-d)\geq0$ for all such $\alpha$. For $d=3$, this is only true for $\alpha\geq 3$ and for $d=2$ it is only true for $\alpha=4$. This proves $\Psi'''\leq 0$ on $(0,1)$ under the stated assumptions on $\alpha$.
	
	\medskip
	
	\emph{Step 3.} We now assume that either $d=2$ and $2\leq \alpha < 4$, or $d=3$ and $2 \leq \alpha < 3$ and show that $\Psi''\geq 0$ on $(0,1)$.
	
	Together with Step 2, this will complete the proof of the lemma.
	
	By continuity (or by a small variation of the following proof), we may assume that $\alpha>2$. By repeated application of \eqref{eq:hyperder}, we have that
	\begin{align*}
		\Psi''(\rho) & = \mathcal C_1 \,   F(2-\tfrac{\alpha}{2} ,\tfrac{6-\alpha-d}{2};\tfrac{d+4}{2};\rho ) - \mathcal C_2 \,  F(2-\tfrac{\beta}{2} ,\tfrac{6-\beta-d}{2};\tfrac{d+4}{2};\rho )
	\end{align*}
	with
	\begin{align*}
		\mathcal C_1 & := \frac{(\alpha-2)(d+\alpha-4)(d+\beta-2)}{4 d(d+2)} \frac{\Gamma(d+\beta-2)\Gamma(\frac{d+\alpha}{2}) \Gamma( d + \frac{\alpha}{2}-1)}{\Gamma(d+\alpha-2)\Gamma(\frac{d+\beta}{2}) \Gamma( d + \frac{\beta}{2}-1)} \,, \\
		\mathcal C_2 & := \frac{ (2-\beta)(4-\beta-d)(d+\beta-2)}{4 d(d+2)} \,.
	\end{align*}
	
	By the assumptions on $d$, $\alpha$, and $\beta$, we know that $0<2-\frac{\alpha}{2}<2-\frac{\beta}{2}$, $0<\frac{6-\alpha-d}{2}<\frac{6-\beta-d}{2}$ and $\frac{d+4}{2}>2-\frac{\beta}{2} + \frac{6-\beta-d}{2}$. Moreover, $\mathcal C_1$ and $\mathcal C_2$ are both positive. (Note that $\mathcal C_1$ would vanish if we allowed $\alpha=2$.) Therefore we are in the situation of Lemma \ref{lem:Hypergeom intersect once} and we infer that either $\Psi''$ has no zero in $[0,1]$, or there is a unique zero $\rho_0\in [0,1]$ and, in this case, $\Psi''>0$ on $[0,\rho_0)$ and $\Psi''<0$ on $(\rho_0,1]$.
	
	By Step 1, we know that $\Psi''(1)\geq 0$. Thus, distinguishing the cases $\Psi''(1)>0$ and $\Psi''(1)=0$, we deduce that in either case we have $\Psi''>0$ on $[0,1)$. This proves the assertion made at the beginning of this step and completes the proof of the lemma.
\end{proof}

Finally, we discuss the function \eqref{eq:goalconvexapp} on the interval $[1,\infty)$.

\begin{lemma}\label{convexlarge}
	Let $d \geq 2$, $2 \leq \alpha \leq 4$ and $\beta_*(\alpha)\leq \beta<-d+4 $. Then the function \eqref{eq:goalconvexapp} is convex on $[1,\infty)$.
\end{lemma}

\begin{proof}
	It suffices to show that 
	\begin{align*}
		\Psi''(\rho) & = \beta^{-1} \frac{\psi_{\beta}'(1)}{\psi_{\alpha}'(1)} \psi_{\alpha}''(\rho) -  \beta^{-1}\psi_{\beta}''(\rho)\\
		& = \beta^{-1} \frac{\psi_{\beta}'(1)}{\psi_{\alpha}'(1)} \frac{\alpha}{2} \Big( \frac{\alpha}{2} - 1 \Big) \rho^{\frac{\alpha}{2}-2}  F(2-\tfrac{\alpha}{2} ,\tfrac{2-\alpha-d}{2};\tfrac{d}{2};\rho^{-1} ) \\
		& \quad - \beta^{-1} \frac{\beta}{2} \Big( \tfrac{\beta}{2} - 1 \Big) \rho^{\frac{\beta}{2}-2}  F(2-\tfrac{\beta}{2} ,\tfrac{2-\beta-d}{2};\tfrac{d}{2};\rho^{-1} )
	\end{align*}
	is nonnegative on $(1,\infty)$. This is equivalent to showing that
	\begin{align*}
		g(\rho) & := \beta^{-1} \frac{\psi_{\beta}'(1)}{\psi_{\alpha}'(1)} \frac{\alpha}{2} \Big( \frac{\alpha}{2} - 1 \Big) \rho^{\frac{\alpha - \beta}{2}}  F(2-\tfrac{\alpha}{2} ,\tfrac{2-\alpha-d}{2};\tfrac{d}{2};\rho^{-1} )  \\
		& \quad - \beta^{-1} \frac{\beta}{2} \Big( \frac{\beta}{2} - 1 \Big)  F(2-\tfrac{\beta}{2} ,\tfrac{2-\beta-d}{2};\tfrac{d}{2};\rho^{-1} )
	\end{align*}
	is nonnegative on $(1, \infty)$.
	
	Since, by \eqref{eq:hyperder},
	$$ 
	\frac{d}{dz}  F(2 - \tfrac{\alpha}{2}, \tfrac{2 - \alpha - d}{2}; \tfrac{d}{2}; z) = \frac{(4- \alpha)(2 - \alpha - d)}{2d}\, F(3 - \tfrac{\alpha}{2}, \tfrac{4 - \alpha - d}{2}; \tfrac{d+2}{2}; z) \,,
	$$
	and $\frac{4 - \alpha -d}2 < 0 < 3 - \frac\alpha 2 \leq \frac{d+2}2$, Lemma \ref{hyperpos} tells us this derivative is nonpositive on $[0,1)$ (note that $2-\alpha-d\leq 0$ under our assumptions). Taking $z=\rho^{-1}$ we therefore have that 
	$$
	 F(2 - \tfrac{\alpha}{2}, \tfrac{2 - \alpha - d}{2}; \tfrac{d}{2}; \rho^{-1})
	$$
	is a nondecreasing function on $(1, \infty)$. Again, from Lemma \ref{hyperpos}, we see that it is also nonnegative.
	
	Likewise, since, by \eqref{eq:hyperder},
	$$
	\frac{d}{dz}  F(2 - \tfrac{\beta}{2}, \tfrac{2 - \beta - d}{2}; \tfrac{d}{2}; z) = \frac{(4- \beta)(2 - \beta - d)}{2d} \,  F(3 - \tfrac{\beta}{2}, \tfrac{4 - \beta - d}{2}; \tfrac{d+2}{2}; z) \,,
	$$
	and $0<\frac{4 - \beta -d}2 < \frac{d+2}2$ (since $-d+4>\beta\geq\beta_*(\alpha)\geq -d+3$), Lemma \ref{hyperpos} tells us this derivative is nonpositive on $[0,1)$. Here we use the fact that the first part of the proof of this lemma only requires that $c > b > 0$. Moreover, we used $2 - \beta - d < 0$. Taking $z=\rho^{-1}$, we therefore have that
	$$
	F(2 - \tfrac{\alpha}{2}, \tfrac{2 - \alpha - d}{2}; \tfrac{d}{2}; \rho^{-1})
	$$
	is a nondecreasing function on $(1, \infty)$.
	
	Since $d \geq 2$, $2 \leq \alpha \leq 4$ and $3 -d < \beta < 4-d$, a quick check shows that $\beta^{-1} \frac{\psi_{\beta}'(1)}{\psi_{\alpha}'(1)} \frac{\alpha}{2} \Big( \frac{\alpha}{2} - 1 \Big)$ and  $- \beta^{-1} \frac{\beta}{2} \Big( \frac{\beta}{2} - 1 \Big)$ are both nonnegative, and $\rho^{\frac{\alpha - \beta}{2}}$ is a positive and increasing function on $(1, \infty)$. Thus $g$ is an increasing function on $(1, \infty)$, and we have our claim.
\end{proof}

We are finally in position to complete the proof of the main result of this section.

\begin{proof}[Proof of Proposition \ref{psiconvex} for $\beta<-d+4$]
	It follows from Lemmas \ref{convexsmall} and \ref{convexlarge} that $\Psi$ is convex on each one of the two intervals $[0,1]$ and $[1,\infty)$. Since $\Psi$ is continuously differentiable at the point $\rho=1$, we deduce that it is convex on all of $(0,\infty)$, as claimed.
\end{proof}

%%%%%%%%%%%%%%%%%%%%%%%%%%%%%%%%%%%%%%%
%%%%%%%%%%%%%%%%%%%%%%%%%%%%%%%%%%%%%%%

\section{Logarithmic case}\label{Sec:Log Case}

In this final section, we explain how the above proofs of Theorems \ref{main1} and \ref{main} can be extended to the case $\beta=0$, where the expression $\beta^{-1}|x-y|^\beta$ in the energy functional is interpreted as $\ln|x-y|$. The basic idea is that the conditions in Lemma \ref{riesz}, which have been verified for $\beta\neq 0$, extend by continuity to the case $\beta=0$.

\begin{proof}[Proof of Theorem \ref{main1} for $\beta=0$]
	We recall the definition of $\psi_\gamma$ from \eqref{eq:defpsigamma}. From the definition of the hypergeometric series it follows that $\psi_0\equiv 1$, and it is easy to see that $\tilde\psi_0 := \lim_{\gamma\to 0} \gamma^{-1}(\psi_\gamma -1)$ exists. In particular, for any $R>0$, we have
	\begin{equation}
		\label{eq:limmain1}
		\lim_{\beta\to 0} \frac{R^\beta}{\beta}\psi_\beta\Big(\Big|\frac{x}{R} \Big|^2\Big) = \ln R + \tilde\psi_0 \Big(\Big|\frac{x}{R} \Big|^2 \Big) \,.
	\end{equation}
	Thus, it follows from \eqref{eq:totpotmain} that the total potential of $(|\Sph^{d-1}|R^{d-1})^{-1} \delta_{\partial B_R(0)}$ is equal to
	$$
	\frac{1}{|\Sph^{d-1}|}\int_{\Sph^{d-1}} \left( \frac{|x-R\omega|^\alpha}{\alpha} - \ln |x-R\omega| \right) d\omega = \frac{ R^{\alpha}}{\alpha} \psi_\alpha \Big(\Big|\frac{x}{R} \Big|^2\Big) - \ln R - \tilde\psi_0 \Big(\Big|\frac{x}{R} \Big|^2 \Big) \,.
	$$
	The function $\tilde\psi_0$ is differentiable, as can be seen by studying the convergence behavior of the corresponding power series at the endpoint of the radius of concergence; see also Remark \ref{limrem} below. Thus, setting the radial derivative of the total potential equal to zero, we arrive at the condition
	$$
	\frac{ R^{\alpha}}{\alpha} \psi_\alpha'(1) - \tilde\psi_0'(1) = 0 \,,
	$$
	which is satisfied precisely for $R=R_{\alpha,0}$ given by \eqref{eq:defr1}. Here we use the value of $\tilde\psi_0'(1)$, which can be read off from \eqref{eq:psivalues2}.
	
	The inequality that needs to be shown is that, for all $x\in\R^d$,
	$$
	\frac{ R_{\alpha,0}^{\alpha}}{\alpha} \psi_\alpha \Big(\Big|\frac{x}{R_{\alpha,0}} \Big|^2\Big)
	- \ln R_{\alpha,0} + \tilde\psi_0 \Big(\Big|\frac{x}{R_{\alpha,0}} \Big|^2 \Big)
	\geq
	\frac{ R_{\alpha,0}^{\alpha}}{\alpha} \psi_\alpha (1)
	- \ln R_{\alpha,0} + \tilde\psi_0 (1) \,.
	$$
	By \eqref{eq:limmain1} and $\lim_{\beta\to 0} R_{\alpha,\beta}=R_{\alpha,0}$, this inequality is a consequence of inequality \eqref{eq:ineqmain1}.
\end{proof}

\begin{remark}\label{limrem}
	The function $\tilde\psi_0$ can be expressed as a generalized hypergeometric series $\Hypergeom32{\cdot,\cdot,\cdot}{\cdot,\cdot}{\cdot}$, whose definition can be found for instance in \cite[Section 9.14]{GrRy}. This expression is not important for us, but may be useful in other contexts and we give it here. Indeed, by introducing polar coordinates, we find
	\begin{align*}
		& \int_{\Sph^{d-1}} \ln |x-\omega| \,d\omega = \frac12 |\Sph^{d-2}| \int_0^\pi \ln\left( |x|^2 - 2 |x|\cos\theta + 1 \right) \sin^{d-2}\theta\,d\theta \\
		& \quad = \frac12 |\Sph^{d-2}| \int_{-1}^1 \ln\left( |x|^2 - 2 |x| t + 1 \right) (1-t^2)^\frac{d-3}{2} \,dt \\
		& \quad = 2^{d-3} |\Sph^{d-2}|  \int_0^1 \left( 2 \ln(1 + |x|) + \ln(1- \tfrac{4 |x|}{(1+|x|)^2} u) \right) (1-u)^{\frac{d-3}{2}} u^\frac{d-3}{2}\,du \\
		& \quad = 2^{d-3} |\Sph^{d-2}| \Bigg( 2^{3-d} \frac{ \sqrt{\pi}\Gamma( \frac{d-1}{2})}{\Gamma(\frac{d}{2})} \ln(1+|x|) \\
		& \quad \qquad \qquad - \frac{4 |x|}{(1+|x|)^2} \frac{\Gamma( \frac{d-1}{2}) \Gamma(\frac{d+1}{2})}{\Gamma(d)} \Hypergeom32% 
		{ 1, 1,  \frac{d+1}{2}}% 
		{2, d}{\frac{4 |x|}{(1+ |x|)^2}} \Bigg)\\
		& \quad = |\Sph^{d-1}| \Bigg(  \ln(1+|x|) - \frac{|x|}{(1+|x|)^2} \Hypergeom32% 
		{ 1, 1,  \frac{d+1}{2}}% 
		{2, d}{\frac{4 |x|}{(1+ |x|)^2}} \Bigg) \,.
	\end{align*}
	Here we have successively changed variables $\cos\theta=\omega\cdot x/|x|$, $t=\cos\theta$ and $u= (1+t)/2$. For the evaluation of the integral we applied the following integral formula (see e.g.\ \cite[eq. (4.1.2)]{Sl}): For $b_0 > a_0 >0$ and $z \in [0,1)$,
\begin{equation} \label{eq:EIF32}
  \Hypergeom32% 
		{ a_0, a_1,  a_2}% 
		{b_0, b_1}{z}  =\frac{\Gamma(b_0)}{\Gamma(a_0)\Gamma(b_0-a_0)}
  \int_0^1 u^{a_0-1}(1-u)^{b_0-a_0-1}
  F(a_1,a_2;b_1;zu) \,d u \,.
\end{equation}
Setting $a_1 = a_2 = 1$ and $b_1=2$
and using that $F(1,1; 2; z)=-\frac{\log(1-z)}{z}$, we then have 
\begin{equation} \label{eq:EIF32log}
  		 \Hypergeom32% 
		{ a_0, 1,  1}% 
		{b_0, 2}{z}
  =-\frac{\Gamma(b_0)}{z\Gamma(a_0)\Gamma(b_0-a_0)}
  \int_0^1 u^{a_0-2}(1-u)^{b_0-a_0-1}
  \log(1-zu) \, d u \,.
\end{equation} 
In our setting, $z = \frac{4 |x|}{(1 + |x|)^2}$. Thus we see that, in the notation of the previous proof,
	$$
	\tilde\psi_0(\rho) = \ln(1+\sqrt{\rho})  - \frac{ \sqrt{\rho}}{(1+\sqrt{\rho})^2}  \Hypergeom32% 
	{ 1, 1,  \frac{d+1}{2}}% 
	{2, d}{\frac{4 \sqrt{\rho}}{(1+ \sqrt{\rho})^2}} \,.
	$$
	Moreover, using the formula of the derivative of a generalized hypergeometric function we find, after a tedious, but straightforward computation,
	$$
	\tilde\psi_0'(\rho) = \frac12 \left( \frac{1}{ \sqrt{\rho}(1+ \sqrt{\rho})}  + \frac{ \sqrt{\rho}-1}{\sqrt{\rho} (1+ \sqrt{\rho})^3} F \left (1  ,   \tfrac{d+1}{2}; d;  \tfrac{4 \sqrt{\rho}}{(1+\sqrt{\rho})^2} \right) \right).
	$$
	Using a transformation formula for the hypergeometric series \cite[(9.134.2)]{GrRy}, we can write this as
	$$
	\tilde\psi_0'(\rho)
	= \begin{cases}
		\tfrac{d-2}{2d}\, F(1,\tfrac{4-d}2;\tfrac d2+1;\rho) & \text{if}\ \rho<1 \,,\\
		\tfrac12 \, \rho^{-1} F(1,\tfrac{2-d}{2};\tfrac d2;\rho^{-1}) & \text{if}\ \rho>1 \,.
	\end{cases}
	$$
	As in the proof of Lemma \ref{hypercont} we can deduce from the latter formula the continuity of $\tilde\psi_0'$ at $\rho=1$, which was used in the previous proof.
\end{remark}

\begin{proof}[Proof of Theorem \ref{main}]
	Let us set
	$$
	\tilde\Phi(x) := - \int_{|y|<1} \ln |x-y| \, (1-|y|^2)^\frac{2-d}{2}\,dy \,,
	$$
	so that, by scaling, the total potential of the measure $\nu$ with
	$$d\nu(y) = (R^2-|y|^2)^\frac{2-d}{2} \mathbbm{1}_{B_R(0)}(y) dy$$	
	is
	$$
	R^2 \left( C_0 \ln R + \tilde\Phi \Big(\frac xR\Big) \right) + \frac{C_0}{2} R^4 \left( \Big| \frac xR \Big|^2 + \frac{d}{4} \right).
	$$
	Here we recall that the constant $C_\beta$ is defined in \eqref{eq:defca} and we use the expression for it in \eqref{eq:defcacomp} with $\beta=0$. The potential of the quadratic ($\alpha = 2$) part of the kernel was computed similarly to the case $\beta\neq 0$.
	
	We will denote the potential $\Phi$ in \eqref{eq:defphimain} by $\Phi_\beta$, making the $\beta$-dependence explicit. A short computation shows that for each $x\in\R^d$,
	\begin{equation}
		\label{eq:limitmain}
		\lim_{\beta\to 0} \left( \Phi_\beta(x) + \beta^{-1}C_\beta \right) = \tilde\Phi(x) - \tilde C_0 \,.
	\end{equation}
	Here we have used again \eqref{eq:defcacomp} and we have set $\tilde C_0 := \frac{d}{d\beta}|_{\beta=0} C_\beta$.
	
	As we have shown above in the proof of Theorem \ref{main} for $\beta\neq 0$, the function
	$$
	x \mapsto R_{2,\beta}^2 \Phi_\beta \Big( \frac x {R_{2,\beta}} \Big) + \frac{C_\beta}2 R_{2,\beta}^{4-\beta} \left( \Big| \frac x{R_{2,\beta}} \Big|^2 + \frac{d}{4-\beta} \right)
	$$
	is constant on $\{x \in \mathbb{R}^d : |x|\leq R_{2,\beta}\}$ and not smaller than this constant in $\{x \in \mathbb{R}^d: |x|> R_{2,\beta}\}$. Clearly, these properties are preserved if we add the constant $R_{2,\beta}^2 \beta^{-1}C_\beta$. Since $\beta\mapsto R_{2,\beta}$ is continuous, these properties are also preserved in the limit $\beta\to 0$. Thus, according to \eqref{eq:limitmain}, we have shown that the function
	$$
	x \mapsto R_{2,0}^2 \left( \tilde\Phi \Big( \frac x {R_{2,0}} \Big) - \tilde C_0 \right) + \frac{C_0}2 R_{2,0}^{4} \left( \Big| \frac x{R_{2,0}} \Big|^2 + \frac{d}{4} \right)
	$$
	is constant on $\{x \in \mathbb{R}^d : |x|\leq R_{2,0}\}$ and not smaller than this constant in $\{x \in \mathbb{R}^d : |x|> R_{2,0}\}$. Clearly, this property is preserved if we add the constant $R_{2,0}^2 ( C_0 \ln R_{2,0} + \tilde C_0)$. This shows that the conditions in Lemma \ref{riesz} are satisfied and we conclude as in the case $\beta\neq 0$.
\end{proof}

%%%%%%%%%%%%%%%%%%%%%%%%%%%%%%%%%%
%%%%%%%%%%%%%%%%%%%%%%%%%%%%%%%%%%

\bibliographystyle{amsalpha}

\end{document}